\documentclass[]{amsart}  

\usepackage{amsthm}
\usepackage{amssymb}
\usepackage{amsfonts}
\usepackage{verbatim}
 \usepackage[colorlinks,backref]{hyperref}

\newcommand{\decone}[1]{{\bf d}_{#1}\mathcal{A}}
\newcommand{\Arr}{\mathcal{A}}
 \newcommand{\Span}[1]{\mathrm{Span}\left\{#1\right\}}

\newcommand{\arr}{\mathcal{A}}

\newcommand{\carr}{\mathbf{c}\mathcal{A}}

\newcommand{\CC}{\mathbb{C}}

\newcommand{\CP}{\mathbb{CP}}

\newcommand{\ZZ}{\mathbb{Z}}
\newcommand{\Hom}{\text{Hom}}

\newcommand{\FF}{{\mathbb F}}

\theoremstyle{plain}

\newtheorem{theorem}{Theorem}[section]
\newtheorem{maintheorem}{Main Theorem}
\newtheorem{prop}{Proposition}[section]

\newtheorem{cor}{Corollary}[section]
\newtheorem{corollary}{Corollary}[section]

\newtheorem{lemma}{Lemma}[section]

\theoremstyle{definition}
\newtheorem{definition}{Definition}[section]
\newtheorem{defn}{Definition}[section]
\newtheorem{example}{Example}[section]

\newtheorem{notation}{Notation}[section]

\newtheorem{remark}{Remark}[section]

\makeatletter 
\let\c@theorem=\c@thm
\let\c@lem=\c@thm
\let\c@lemma=\c@thm
\let\c@prop=\c@thm
\let\c@cor=\c@thm
\let\c@cong=\c@thm
\let\c@defn=\c@thm
\let\c@remark=\c@thm
\let\c@example=\c@thm
\let\c@note=\c@thm
\let\c@nte=\c@thm
\let\c@observe=\c@thm
\let\c@notation=\c@thm
\makeatother

 \usepackage{amsmath}
 \usepackage{amscd}
 \usepackage{ifthen}

 \usepackage{todonotes}
 \usepackage{tikz}
 \usetikzlibrary{arrows,chains,matrix,positioning,scopes}

 \makeatletter
 \tikzset{join/.code=\tikzset{after node path={%
 \ifx\tikzchainprevious\pgfutil@empty\else(\tikzchainprevious)%
 edge[every join]#1(\tikzchaincurrent)\fi}}}
 \makeatother
 
 \tikzset{>=stealth',every on chain/.append style={join},
          every join/.style={->}}
 
 \usepgflibrary{shapes.geometric}

\begin{document}

\title[Converse to Oka and Sakamoto]{A converse to a theorem of Oka and Sakamoto for complex line arrangements }
\author{Kristopher Williams}
\address{Department of Mathematics, Doane College, Crete, NE 68333, USA}
\email{kristopher.williams@doane.edu}

\subjclass[2010]{Primary 14F35; Secondary 57M05, 52C35, 32s22.}

\keywords{line arrangement, hyperplane arrangement, Oka and Sakamoto,
  direct product of groups, fundamental groups, algebraic curves}

\begin{abstract} Let  \( C_1 \) and  \( C_2 \) be algebraic plane curves in  \( \CC^2 \) such that the curves intersect in  \( d_1 \cdot d_2 \) points where  \( d_1,d_2\) are the degrees of the curves respectively. Oka and Sakamoto proved that  \( \pi_1(\mathbb{C}^2 \setminus C_1 \cup C_2)) \cong   \pi_1(\mathbb{C}^2 \setminus C_1 ) \times  \pi_1(\mathbb{C}^2 \setminus C_2) \) \cite{Oka-Sakamoto-ProductTheorem-MR513072}. In this paper we prove the converse of Oka and Sakamoto's result for line arrangements. Let  \( \mathcal{A}_1 \) and  \( \mathcal{A}_2 \) be non-empty arrangements of lines in  \(\mathbb{C}^2 \) such that \(
  \pi_1(M(\mathcal{A}_1 \cup \mathcal{A}_2)) \cong \pi_1(M(\mathcal{A}_1)) \times \pi_1(M(\mathcal{A}_2)).
  \)
  Then,  the intersection of \( \mathcal{A}_1 \) and  \( \mathcal{A}_2 \) consists of  \( |\mathcal{A}_1| \cdot |\mathcal{A}_2| \) points of multiplicity two.
 \end{abstract} 
\maketitle

\section{Introduction}\label{sec:intro}
Let  \( V \) be a hypersurface in  \( \CP^l \). By the hyperplane section theorem of Hamm and Le \cite{MR0401755}, the fundamental group  \( \pi_1(\CP^l \setminus V) \) is isomorphic to the fundamental group  \( \pi_1(\CP^2 \setminus C) \) where  \( \CP^2 \) is a generic 2-dimensional projective subspace in  \( \CP^l \) and  \( C= V \cap \CP^2 \). Zariski  began the first systematic study of the fundamental group of the complement of the curve  \( \pi_1(\CP^2 \setminus C) \)  in 1929 \cite{Zariski-existence}. Since then, many authors have studied the properties of the fundamental group of complements of hypersurfaces. 

The complements of hyperplane arrangements are one of several classes of hypersurfaces whose fundamental groups are a subject of active research. A hyperplane arrangement  \( \arr\) is a finite collection of codimension one affine subspaces in  \( \CC^l \) (see the work of Orlik and Terao \cite{OT-Arrs-MR1217488} as a general reference for arrangements). One of the objects associated to an arrangement is the complement of the arrangement  \( M(\arr) = \CC^l \setminus \cup_{H \in \arr} H \). Similarly, one may define a projective arrangement  \(  \arr^* \) as a finite collection of projective hyperplanes in projective space with complement  \( M(\arr^*) = \CP^l \setminus \cup_{H \in \arr} H  \). Let  \( H'  \) be any hyperplane in the projective arrangement  \( \arr^* \). If we consider  \( H' \) as the ``hyperplane at infinity,'' then we may identify  \( \CP^l \setminus H' \) with affine space  \( \CC^l \). Also, the set  \( \decone{H'}^*=\{  H \setminus H' | H \in \arr^* \} \) is identified with a set of hyperplanes in  \( \CC^l \). While the resulting arrangement \( \decone{H'}^* \) certainly depends on the choice of hyperplane  \( H' \), the complements of the arrangements  \( M(\decone{H'}^*) \) are diffeomorphic for all choices of  \( H' \in \arr^* \) (Proposition 5.1, \cite{OT-Arrs-MR1217488}).

Another interesting object associated to an arrangement  \( \arr \) is the intersection lattice  \( L(\arr) \), which is the set of non-empty intersections of hyperplanes in the arrangement. The intersection lattice is a partially ordered set ordered by reverse inclusion. Any property of the arrangement or its complement that may be determined from the intersection lattice is called combinatorial.

One of the major questions in the study of hyperplane arrangements is what extent the combinatorics of the arrangement determines the topology of the complement of the arrangement. A result of Orlik and Solomon \cite{Orlik-Solomon-MR558866} shows that the intersection lattice determines the cohomology algebra of the complement of an arrangement. However, Rybnikov \cite{Ryb98} gave an example of a pair of arrangements that had the same intersection lattice but whose complements have non-isomorphic fundamental groups. 

While the fundamental group is not a combinatorial invariant, it is interesting to ask what properties of the fundamental group are combinatorial and for what classes of combinatorics the fundamental group is determined. In 1996, Fan \cite{Fan-Directproducts-MR1460414} defined a combinatorially determined graph associated to any projective arrangement in  \( \CP^2\). Fan went on to show that if the graph is a forest of  trees then the fundamental group of the arrangement is isomorphic to a direct product of free groups. The converse to this theorem was later proven by Eliyahu, Liberman, Schaps and Teicher \cite{ELST-Char-DSFG-converse-AG}. In fact, if the fundamental group of the complement of an arrangement in  \( \CC^2 \) is isomorphic to direct product of free groups, then the fundamental group determines the homotopy type of the complement of the arrangement \cite{KW-DPFG}. 

The following is a combinatorial theorem proven by Oka and Sakamoto regarding the fundamental groups of complements of plane algebraic curves.

\begin{theorem}[\cite{Oka-Sakamoto-ProductTheorem-MR513072}]
\label{thm:okasakamoto}
Let $C_{1}$ and $C_{2}$ be plane algebraic curves in $\CC^{2}$. Assume that the
intersection $C_{1}\cap C_{2}$ consists of distinct $d_{1}\cdot d_{2}$ points where $d_{i}(i=1,2)$ are the respective degrees of $C_{1}$ and $C_{2}$. Then the fundamental group $\pi_{1}(\CC^{2}\setminus C_{1}\cup C_{2})$ is isomorphic to the product of $\pi_{1}(\CC^{2} \setminus C_{1})$ and $\pi_{1}(\CC^{2}\setminus C_{2})$ .
\end{theorem}

The converse of Theorem \ref{thm:okasakamoto} is not true in general. Consider the curve  \( C_1 \) defined as the zero locus of  \( Q_1(x,y)=x \) and the curve \( C_2  \) defined as the zero locus of  \( Q_2(x,y)=y^2-x^2-x^3 \). Then  \( C_1 \) and  \( C_2 \) intersect in a single point at the origin in a point of multiplicity three. However, the fundamental group is a direct product of groups: \[ 
\pi_1(\CC^2 \setminus C_1 \cup C_2) \cong \pi_1(\CC^2 \setminus C_1) \times \pi_1(\CC^2 \setminus C_2) \cong \ZZ \times \ZZ. 
\] 

The natural question is for what families of curves does the converse of Theorem \ref{thm:okasakamoto} hold. In this paper, we prove the following theorem, which is similar to the converse of Theorem \ref{thm:okasakamoto} for hyperplane arrangements. This theorem yields combinatorial information about the arrangement from a non-combinatorial invariant.

\begin{maintheorem}
  Let  \( \arr^*\) be an arrangement of projective lines in  \( \CP^2 \) such that \( \pi_1(M(\arr^*)) \cong G_1 \times G_2 \), where both  \( G_1 \) and  \( G_2 \) are non-trivial groups. Then there exists a line  \( L \in \arr^* \) such that the decone \( \decone{L}^* \) in  \( \CC^2 \) has non-trivial subarrangements  \( \mathcal{C}_1, \mathcal{C}_2 \) such that  \( \decone{L}^* = \mathcal{C}_1 \cup \mathcal{C}_2 \) and the curves defined by  \( \mathcal{C}_1\) and \( \mathcal{C}_2 \) intersect transversely. 
\end{maintheorem}

Finally, we prove the converse of Oka and Sakamoto's theorem for line arrangements:

\begin{maintheorem}
 Let  \( \arr_1 \) and  \( \arr_2 \) be empty arrangements in  \( \CC^2 \) such that \[
  \pi_1(M(\arr_1 \cup \arr_2)) \cong \pi_1(M(\arr_1)) \times \pi_1(M(\arr_2)).
  \]
  Then,  the intersection of \( \arr_1 \) and  \( \arr_2 \) consists of  \( |\arr_1| \cdot |\arr_2| \) points of multiplicity two.
\end{maintheorem}

\section{Preliminaries}
In this section we recall some facts about hyperplane arrangements. We also prove new results for the class of arrangements with fundamental groups that are isomorphic to non-trivial direct products of groups. 
 
\subsection{Arrangement Properties}
Let  \( \arr = \{H_i\}_{i=1}^{n} \) be an affine arrangement of  \( n \) distinct hyperplanes in  \( \CC^l \). By choosing a system of system of coordinates  \( \CC[z_1,z_1,\cdots,z_{l}] \) we may choose a degree one polynomial  \( \alpha_i \) (defined up to multiplication by a non-zero constant  \( c \in \CC \)) for each hyperplane \( H_i \) such that  \( H_i \) is the zero locus of  \( \alpha_i. \)  Denote by  \( Q(\arr) = \prod\limits_{i=1}^{n} \alpha_i\) the product of the polynomials associated to each hyperplane. The polynomial  \( Q(\arr) \) is called a \textbf{defining polynomial for  \( \arr \)} and is unique up to non-zero scalar multiple.

Let  \( Q' \in \CC[z_0,z_1,\cdots,z_l] \) be the polynomial  \( Q(\arr) \) homogenized with respect to  \( z_0 \) and define a polynomial  \( Q(\carr) = z_0 Q' \). The arrangement defined by  \( Q(\carr ) \) is called the \textbf{cone} over  \( \arr \) and is denoted by  \( \carr \). Note that  \( \carr \) is an arrangement of hyperplanes in  \( \CC^{l+1} \). This procedure has an inverse operation referred to as \textbf{deconing}. Let   \( \mathcal{C} \) be a  central (the intersection of all of the hyperplanes in the arrangement is non-empty) arrangement. Choose a hyperplane  \( L \in \mathcal{C}  \) and choose coordinates  \( \CC[z_0,z_1,\cdots,z_l] \) so that  \( L  \) is the zero locus of the polynomial  \( \alpha_L = z_0 \). If  \( Q(\mathcal{C} ) \) is a defining polynomial for  \( \mathcal{C} \) then we may define a polynomial  \( Q(\mathbf{d}_{L}\mathcal{C}) \in \CC[z_1,\cdots,z_l] \) by letting  \( z_0 \) equal  \( 1 \) in \( Q(\mathcal{C} ) \). Geometrically, one may view the deconing of an arrangement as the projectivization of the arrangement followed by removing the image of the hyperplane  \( L \) (the ``hyperplane at infinity''), then identifying the complement with affine space. 

The following proposition relates an arrangement and the cone over that arrangement.

\begin{prop}[Proposition 5.1, \cite{OT-Arrs-MR1217488}]\label{prop:hopf}
  Let  \( \arr \) be an affine arrangement and let  \( \carr \) be the cone of the arrangement  \( \arr \). The Hopf bundle  \( p\colon \CC^{l+1} \setminus \{0\} \to \CP^{l} \) is the map with fiber  \( \CC^* \) which identifies  \( z \in \CC^{l+1}  \) with  \( \lambda z \in \CC^* \).  The restriction of the map  \( p: M(\carr) \to M(\arr) \) is a trivial bundle, so \[
  M(\carr) \cong M(\arr) \times \CC^*.
  \]
\end{prop}

We introduce the following definition for the sake of brevity.
 
\begin{defn}
  If an arrangement $\arr$ in  \( \CC^2 \) is the union of two nontrivial subarrangements $\arr_1$ and $\arr_2$ with  such that $\arr_1$ and $\arr_2$ intersect in exactly  \( |\arr_1 |\cdot |\arr_2|  \) points of multiplicity two, then we say that $\arr = \arr_1 \cup \arr_2$ is a \textbf{general position partition} of the arrangement. 
\end{defn}

We use the phrase ``general position'' as two distinct lines in the complex plane intersect transversely in a point of multiplicity two or have no point of intersection.

\subsection{Fundamental Group}
From Proposition \ref{prop:hopf} and properties of fundamental groups of spaces, the next lemma follows easily.

\begin{lemma}\label{lem:fundgroup-cone-decone}
  The fundamental group of the cone of an arrangement  \( \arr \) and the fundamental group of the arrangement are related by \[
  \pi_1(M(\carr)) \cong \pi_1(M(\arr)) \times \ZZ.
  \]
\end{lemma}

As  \( Q(\carr) \in \CC[z_0,z_1,\cdots,z_l] \) is homogeneous, the polynomial also defines an arrangement of projective hyperplanes in  \( \CP^l \), which we will denote by  \( \arr^* \). In the proof of Proposition \ref{prop:hopf}, Orlik and Terao show that  \( M(\arr^*) \in \CP^l \) is diffeomorphic to  \( M(\arr) \in \CC^{l} \).  We then have the following Lemma.

\begin{lemma}\label{lem:fungroups-rel}
  Let  \( \arr \) be an arrangement,  \( \carr \) the cone of the arrangement  \( \arr \) and  \( \arr^* \) the projective arrangement associated to  \( \arr \). Then \[
\pi_1(M(\arr^*) ) \cong \pi_1(M(\arr)), \]
and \[
\pi_1(M(\carr) ) \cong \pi_1(M(\arr^*)) \times \ZZ.
  \]\end{lemma}

\subsection{Graphs of Fan Type}
In this paper, we will use a graph defined by Fan in \cite{Fan-Directproducts-MR1460414}. We recall the definition of the graph, and then give several theorems that will be useful later.

\begin{definition}
  Let  \( \arr^* = \{L_1, L_2, \dots, L_n\} \) be an arrangement of  \( n \) distinct projectives lines in  \( \CP^2 \). Let  \( M \) denote the set of points in  \( \CP^2 \) where two or more lines from  the arrangement \( \arr^*  \) intersect and let  \( M_3 \) be the subset of  \( M \) consisting of points with multiplicity greater than or equal to three.

  Define a graph  denoted by \( F(\arr) \) called a \textbf{graph of Fan type} of the arrangement \( \arr^* \) as follows. 
  \begin{itemize}
  \item Let the set of points  \( M_3 \) be the vertices of  \( F(\arr) \).
  \item  For each line  \( L_i \in \arr^* \), let  \( S_i = M_3 \cap L_i \). If the set  \( S_i \) is not empty, then choose an ordering of the points in  \( S_i \) given by  \( S_i = \{p_1, p_2, \cdots, p_m\}\). For each  \( j \in \{1,\dots,m-1\} \), choose a simple arc  \( a_j \) in  \( L_i \) that connects  \( p_j \) to  \( p_{j+1} \), avoids all points in  \( M \), and avoids all arcs previously chosen. Let  \( A_i  \) be the set of simple arcs chosen for the line  \( L_i \). The edges of  \( F(\arr) \) will consist of  the set of arcs \( A_1 \cup A_2 \cup \cdots \cup A_n \). 
  \end{itemize}
\end{definition}

The vertices of graphs of Fan type are uniquely defined. However, the edges are not uniquely defined since any line containing more than two multiple points admits many orderings of those points which leads to different sets of edges. This situation motivates the following definition.

\begin{definition}
  Let  \( \mathcal{F}(\arr^*) \) denote the set of all possible graphs of Fan type for the arrangement  \( \arr^*. \)
\end{definition}

We collect some useful properties of graphs of Fan type of an arrangement $\arr^*$ that imply  \( \arr^* \) has a decone with a general position partition.

\begin{lemma}\label{lem:disconnected}
  Let $\arr^*$ be a projective line arrangement in $\CP^2$. If a graph of Fan type  \( F \in \mathcal{F}(\arr^*) \) is disconnected, then  \( \arr^* \) has a decone with a general position partition.\end{lemma}
\begin{proof}
  Let $C$ be a connected component of the graph $F$ and let $\mathcal{C}$ denote the set  of all lines in  \( \arr^* \) that contain a vertex in  \( C \). Let \(\mathcal{D} = \arr^* \setminus \mathcal{C} \). The set $\mathcal{D}$ is non-empty as otherwise the graph  \( F \) would be connected.

  Let  \( H_c \in \mathcal{C} \) and  \( H_D \in \mathcal{D} \). Suppose that  \( H_C \) and  \( H_D \) intersect in a point of multiplicity greater than two. (As the lines are in projective space, the multiplicity of their intersection is at least two.) As  \( H_C \in \mathcal{C}\), then by definition the line  \( H_C \) contains a vertex in the connected component  \( C \) of the graph  \( F \). By definition of graphs of Fan type, this means either  \( H_C \cap H_D \) is a vertex of the component  \( C \) or there is a path from another vertex in  \( C \) contained in  \( H_C \) to \( H_C \cap H_D \). In either case, this implies that  \( H_D \) contains a vertex in the connected component  \( C \); therefore,  \( H_D \) is in the set  \( \mathcal{C} \) which contradicts the fact that  \( H_D \in \mathcal{D} \). Therefore,  \( H_C \) and \(H_D\) intersect in a point of multiplicity two.

  Let  \( L \) be an arbitrary line in the arrangement  \( \mathcal{C} \). Then the decone  \( \decone{L} \) has two components  \( \mathbf{d}_L\mathcal{C} \)  and  \(\mathbf{d}_L\mathcal{D} \), the images of  \( \mathcal{C} \) and  \( \mathcal{D} \) under the decone operation. As the arrangements \( \mathcal{C} \) and  \( \mathcal{D} \)  were in general position in projective space, their images are in general postion in affine space after the decone. Therefore,  \( \arr^* \) has a decone with a general position partition.
\end{proof}

\begin{lemma}\label{lem:multi-pt}
  Let  \( \arr^* \)  be a projective arrangement of lines in  \( \CP^2 \) containing at least three lines. If there is a line $H$ in $\arr^*$ such that all multiple points contained in $H$ are points of multiplicity two, then  \( \arr^* \) has a decone with a general position partition.\end{lemma}
\begin{proof} Let  \( L  \) denote any line in  \( \arr^* \) that is not $H$. Then the decone  \( \decone{L} \) may be written as two subarrangements  \( \mathcal{H}=\{\mathbf{d}_LH\} \) and  \( \decone{L} \setminus \mathcal{H} \). As all lines in $\arr^*$ intersect \( H \) in double points, the image of these lines will also intersect the image of  \( H \) in double points in $\CC^2$. Therefore,  \( \mathcal{H} \) and  \(\decone{L} \setminus \mathcal{H} \) are in general position. Whereby, we conclude that  \( \arr^* \) has a decone with a general position partition.
\end{proof}

\begin{lemma}\label{lem:simple} Suppose that all graphs in $\mathcal{F}(\arr^*)$ are connected and every hyperplane in $\arr^*$ contains a point of multiplicity at least three. If there is a graph $F \in \mathcal{F}(\arr^*)$ such that $F$ has an edge that is not part of a simple circuit, then  \( \arr^* \) has a decone with a general position partition. \end{lemma}
\begin{proof} Recall that a simple circuit is a path in a graph such that the first and last vertices of the path are the same, no vertices are repeated (except for the first vertex as the last vertex) and no edges are repeated.

Let  \( L\) be the line in  \( \arr^* \) containing the edge  \( e \) that is not part of a simple circuit in  \( F \). Let  \( v \) and  \( w \) denote the vertices of the edge  \( e \) . The graph  \( F \setminus \{e\} \) has two connected components. If not, then there is a simple path  \( P \) from  \( v \) to  \( w \). Combining the path  \( P \) with the edge  \( e \) would create a simple circuit containing  \( e \) which is a contradiction.

Denote the components of  \( F \setminus \{e\} \) by  \( F_v \) and  \( F_w \) where  \( F_v \) is the component containing  \( v \) and  \( F_w \) is the component containing  \( w \). Let  \( \mathcal{B}_v \) denote the set of lines in  \( \arr^* \) containing vertices in  \( F_v \) except for  \( L \). Likewise let \( \mathcal{B}_w \) denote the set of lines in  \( \arr^* \) containing vertices in  \( F_w \) except for  \( L \).  Since every line in  \( \arr^* \) contains a higher order multiple point, each line besides  \( L \) must be in either   \( \mathcal{B}_v \) or  \( \mathcal{B}_w \). Therefore,  \( \arr^* = \mathcal{B}_w ~\dot{\cup}~ \mathcal{B}_v ~\dot{\cup}~ \{L\}\). The sets of lines \( \mathcal{B}_v \) and \( \mathcal{B}_w \) are disjoint. If they have a line, $H$, in common, then $H$ would contain a vertex that is connected to  \( v \) in  \( F_v \) and a vertex that is connected to  \( w \) in  \( F_w \). By definition of graphs of Fan type, there is a path between these two vertices, hence the vertices  \( v \) and  \( w \)  are connected which is a contradiction.

Let  \( H_v \in \mathcal{B}_v \) and  \( H_w \in \mathcal{B}_w \) be chosen arbitrarily. Suppose that  \( H_v \) and  \( H_w \) intersect in a point of multiplicity greater than two in  \( \arr^* \). (We know the lines intersect in a point of at least multiplicity two as they are in  \( \CP^2 \).) Then  \( z = H_v \cap H_w \) will be a vertex in all graphs of Fan type. As  \( H_v \in \mathcal{B}_v  \), by definition there is a path from the point  \( v \) to  \( z \) in the graph  \( F \setminus \{e\} \). Likewise, as  \( H_w \in mathcal{B}_w \) there is a path from the point  \( z \) to the  \( w \) in  \( F \setminus \{e\}. \) Combining these paths creates a path from  \( v \) to  \( w \). This yields a contradiction as  \( v \) and  \( w \) are in distinct connected components. Therefore  \( H_v \) and  \( H_w \) intersect in a point of multiplicity two. As these lines were chosen arbitrarily, we see that  \( \mathcal{B}_v  \) and  \( \mathcal{B}_w  \) intersect in general position in  \( \CP^2 \).

  Therefore, the arrangement  \( \decone{L} \) has two subarrangements  \( \mathbf{d}\mathcal{B}_v \) and  \(\mathbf{d} \mathcal{B}_w \) the images of  \( \mathcal{B}_v \) and  \( \mathcal{B}_w \) respectively. These arrangements intersect in general position, therefore  \( \arr^* \) has a decone with a general position partition.
\end{proof}

\begin{cor}
If $\arr^*$ does not have a decone with a general position partition, then for every $F \in \mathcal{F}(\arr^*)$ it follows that $F$ is connected, every edge in $F$ is contained in a simple circuit and every hyperplane in $\arr^*$ contains a vertex in $F$.
\end{cor}

\begin{theorem}\label{thm:norepeats}
  Let  \( \arr^* \) be an arrangement of projective lines. If there is an edge  \( e \) on a line  \( L \) such that for all  \(  F \in \mathcal{F}(\arr^*) \) with  \( e \in F \) every simple circuit containing  \( e \) has at least two edges contained in the line  \( L \), then \( \arr^* \) has a decone with a general position partition.
\end{theorem}
\begin{proof}From Lemma \ref{lem:disconnected}, we know that if any choice of Fan's graph is disconnected, then the conclusion follows. Therefore, we assume that any choice of graph of Fan type from the collection \( \mathcal{F}(\arr^*) \) is connected. Let \(  F \in \mathcal{F}(\arr^*) \) be any choice of graph of Fan type that contains the edge  \( e \) as described in the statement of the theorem. Let  \( v \) and  \( w \) denote the vertices of the edge  \( e \). 

  Let  \( E \) denote the set of edges in  \( F \) that are contained in the line  \( L \). Then, \( F^* = F \setminus E \) is a subgraph of  \( F \) . If  \( F^* \) is connected, then there is a simple path from  \( v \) to  \( w \). However, combining this path with the edge  \( e \) would create a simple circuit in the graph  \( F \), but this contradicts the hypothesis that every simple circuit containing the edge  \( e \) in the graph  \( F \) has at least two edges contained in the line  \( L \) since the path comes from  \( F^*. \) Therefore the graph  \( F^* \) has at least two connected components and  \( w \) and  \( v \) are in different components. 

  Let  \( \mathcal{B}_v \) denote the set of lines in  \( \arr^* \setminus \{L\} \) that contain vertices with paths to  \( v \) in  \( F^* \). Let  \( \mathcal{B}_w =\left( \arr^* \setminus \{L\} \right) \setminus \mathcal{B}_v\).   Both of these arrangements are non-empty as  \( v \) and  \( w \) represent points of multiplicity greater than three. Let  \( H_v  \in \mathcal{B}_v\) and  \( H_w \in \mathcal{B}_w \) be chosen arbitrarily. Suppose  \( H_v \cap H_w \) is a point of multiplicity greater than two. Then by definition of  \( H_v \in \mathcal{B}_v \) there is a path from the vertex  \( H_v \cap H_w \) to  \( v \). Therefore, by definition of the set  \( \mathcal{B}_v \), we have  \( H_w \in \mathcal{B}_v \). However, this is a contradiction as  \( H_w \in \mathcal{B}_w \). Therefore, the lines  \( H_v \) and  \( H_w \) intersect in a point of multiplicity two in the arrangement  \( \arr^* \). As these lines were chosen arbitrarily, the arrangements  \( \mathcal{B}_v \) and  \( \mathcal{B}_v \) are in general position. Using  \( L \) as the line at infinity will result in the arrangement  \( \decone{L}^*   \) that has a general position partition given by  \(  \mathbf{d}_L\mathcal{B}_v \cup \mathbf{d}_L\mathcal{B}_w \).
\end{proof}

\subsection{Characteristic Varieties} 

The characteristic varieties can be defined for any space that is homotopy equivalent to CW-complex with finitely many cells in each dimension \cite{Suciu-arXiv:1111.5803}. In this paper, we restrict out attention to spaces that are complements of hyperplane arrangements. 

Let  \( \arr = \{H_1,\dots, H_n\}\) be an arrangement of hyperplanes in  \( \CC^l \). As the fundamental group  \( \pi_1(M(\arr)) \) has torsion-free abelianization with rank  \( n \), the character variety  \( \Hom(\pi_1(M(\arr)), \CC^*) \) is identified with  \( (\CC^*)^n \). A generating set for a presentation of the fundamental group \( \pi_1(M(\arr)) \) is given by  \( \{\gamma_1, \dots, \gamma_n\}\) where each  \( \gamma_i\) is a meridional loop around  \( H_i \) whose orientation is given by the complex structure.

\begin{definition}[\cite{Suciu-arXiv:1111.5803}]\label{def:charvar}
  The \textbf{characteristic varieties} of the arrangement \( \arr \)  are the cohomology jumping loci of  \( M(\arr) \) with coefficients in the rank 1 local systems over  \( \CC^* \):
  \[
  V_d^i(\arr) : = \{ \mathbf{t} \in (\CC^*)^n | \dim_{\CC} H^i(M(\arr), \CC_{\mathbf{t}}) \geq d \}
  \] where   \( \mathbf{t} = \{t_1, \dots, t_n\} \) determines a representation  \( \pi_1(M(\arr)) \to \CC^* \),  \( \gamma_i \mapsto t_i \) that induces a rank one local system  \( \CC_{\mathbf{t}} \).
\end{definition}

In this paper, we shall only be concerned with the varieties where  \( i=1 \) and  \( d=1 \). The characteristic varieties  \( V_1^1(\arr) \) depends (up to a monomial automorphism of the algebraic torus  \( (\CC^*)^{n} \)) only on the fundamental group  \(G=\pi_1(M(\arr))\) (Subsection 3.1, \cite{Suciu-arXiv:1111.5803}), therefore we will also use the notation \( V_1^1(G) \). 

\subsubsection{Direct Products of Groups}
Let  \( \arr \) denote an arrangement  \( n \) of hyperplanes and let  \( G = \pi_1(M(\arr)) \) denote the fundamental group of the complement of the arrangement. Further suppose that  \( G \cong G_1 \times G_2 \) where  \( G_1 \) and  \( G_2 \) have ranks  \( n_1 \geq 1 \) and  \( n_2 \geq 1\) respectively. As  \( G \) may be finitely presented with rank equal to the number of hyperplanes in the arrangement, it follows that  \( G_1 \) and  \( G_2  \) may also be finitely presented and  \( n = n_1 + n_2 \). The characteristic variety  \( V_1^1(G_i) \) is a subset of the algebraic torus \( (\CC^*)^{n_i}  \).  

\begin{theorem}\label{thm:group-to-charvar-union}
  If $G \cong G_1 \times G_2$ is the fundamental group of an arrangement  \( \arr \) of  \( n \) hyperplanes, then the characteristic varieties $V_1^1(\arr)$ is isomorphic to $$(V_1^1(G_1) \times \mathbf{1}^{n_2}) \cup (\mathbf{1}^{n_1} \times V_1^1(G_2)).$$
\end{theorem}
\begin{proof}
  Let  \( M,M_1\) and  \( M_2 \) be the canonical CW complexes generated by finite presentations for  \( G,G_1 \) and  \( G_2 \) respectively. Then  \( M \) is homotopy equivalent to the product  \( M_1 \times M_2 \). By Theorem 3.2 in \cite{CS-CharVars-MR1692519}, \[
  V_1^1(M_1 \times M_2) = (V_1^1(M_1) \times \mathbf{1}^{n_2}) \cup (\mathbf{1}^{n_1} \times V_1^1(M_2)) \subseteq (\CC^*)^{n_1} \times  (\CC^*)^{n_2}.
  \]
  As the characteristic varieties depend only on the group, there are monomial automorphisms of the algebraic torus \( (\CC^*)^{n} \) such that 
\begin{align*}
V_1^1(\arr) &\cong V_1^1(M) \\
&\cong (V_1^1(M_1) \times \mathbf{1}^{n_2}) \cup (\mathbf{1}^{n_1} \times V_1^1(M_2)) \\
&\cong (V_1^1(G_1) \times \mathbf{1}^{n_2}) \cup (\mathbf{1}^{n_1} \times V_1^1(G_2)).\\
\end{align*}
Thus, the varieties are isomorphic as desired.\end{proof}

\begin{corollary}\label{cor:prod-union}
  Let  \( \arr^* \) be a projective arrangement of lines in  \( \CP^2 \). If \( \pi_1(M(\arr^*)) \cong G_1 \times G_2 \) then $V_1^1(M(\carr))$ is isomorphic to $$(V_1^1(G_1) \times \mathbf{1}^{n_2} \times 1) \cup (\mathbf{1}^{n_1} \times V_1^1(G_2) \times 1) \cup (\mathbf{1}^{n_1 + n_2} \times 1).$$
\end{corollary}
\begin{proof}
  From Lemma \ref{lem:fungroups-rel},  \( \pi_1(M(\carr) ) \cong \pi_1(M(\arr^*)) \times \ZZ \cong G_1 \times G_2 \times \ZZ \). Also,  \( V_1^1(\ZZ) = \{1\} \). Therefore, using Theorem \ref{thm:group-to-charvar-union} twice, the conclusion follows.
\end{proof}

\subsection{Resonance Varieties}
Let $\arr$ be an arrangement in $\CC^l$. Then, the cone of the arrangement $\carr$ is an arrangement in $\CC^{l+1}$. We denote the projectivization of the arrangement in $\CP^l$ by $\arr^*$. The intersection poset is the set of non-empty intersections of hyperplanes in the arrangement and is denoted by \[
L(\arr) := \{ \cap_{H \in \mathcal{B}} H \vert \mathcal{B} \subseteq \mathcal{A} \}.
\] The rank of an element  \( X \in L(\arr) \) is equal to the codimension of the space  \( X \) in  \( \CC^l \). The rank  \( n \) elements are denoted by  \( L_n(\mathcal{A}) \). 

Falk introduced the resonance varieties associated to an arrangement in \cite{Falk-ArrsAndCohom-MR1629681}. We recall the basic notation and ideas here  and direct the interested reader to Falk's original work for more information. Let $A=A(\arr)$ be the graded Orlik-Solomon algebra associated to an arrangement generated by  \( \{a_1,\dots,a_n \} \) where  \( a_i \) is associated to  \( H_i \in \arr \). If we fix an element $\omega \in A^1$, then the map $d_\omega \colon A^p \to A^{p+1}$ defined by left multiplication creates a complex $(A,d_\omega)$ as $d_\omega \circ d_\omega=0$. Notice that $\omega = \sum\limits_{i=1}^n \lambda_i a_i$ where $\lambda_i \in \CC$. Therefore, we associate each $\omega$ with a vector $\lambda \in \CC^n$.

As $(A,d_\omega)$ is a complex, we denote the cohomology of the complex by $H^p(A,\omega)= H^p(A,\lambda)$. Finally, we define the \textbf{resonance varieties} associated to the arrangement by \[ 
R_p(\Arr) = \{ \lambda \in \CC^n | H^p(A(\arr),\lambda) \neq 0\}.
\] 

The following definition follows from Lemma 3.14 in \cite{Falk-ArrsAndCohom-MR1629681}:

\begin{definition} Let $\arr = \{H_i\}_{i=1}^n \in \CC^d$ be a central arrangement of hyperplanes. For each $X \in L_2(\arr)$ with   \( X\) contained in at least three hyperplans in  \( \arr \), the \textbf{local component of $X$} in $R_1(\arr)$ is given by

$$R_1(\arr, X) = \left\{ \lambda \in \CC^n \Big| \sum\limits_{i=1}^n \lambda_i =0 \text{ and } \lambda_i = 0 \text{ if } X \nsubseteq H_i \right\}.$$
We will denote this component by the simpler notation \( R_1^{loc}(X)\) when there is not confusion about the arrangement.
\end{definition}

The next lemma follows easily from the definition as each vertex corresponds to a point in $\arr^*$ of multiplicity at least three, hence a rank two component of $L(\carr)$.

\begin{lemma}
  Let  \( \arr^*\) be a projective arrangement of lines in  \( \CP^2. \) Each vertex in a graph of Fan type of $\arr^*$ induces a non-trivial local component of the resonance varieties in $R_1(\carr)$.
\end{lemma}

\begin{example}
\label{ex:braid-arr-resvars}
Consider the arrangement  \( \arr \) in $\CC^3$ defined by the polynomial  \( Q(x,y,z) = xyz(x+y)(y+z)(x+z) \). This is the rank 3 braid arrangement with associated matroid  \( M(K_4) \). We may depict the real part of the projectivization of this arrangement in Figure \ref{fig:braid-arrangement}.  Using the labeling of the lines as in the figure, we have four local components in  \( R_1(\arr) \). Let  \( \{i,j,k\} \) denote the point in  \( \arr^* \)where the lines  labelled by  \( i,j,k \) intersect. The four local components are
 \begin{align*}
  R_1(\arr, \{1,2,6\})&= \left\{ \lambda \in \CC^6 | \lambda_1 + \lambda_2 + \lambda_6 = 0, \lambda_3=\lambda_4=\lambda_5=0 \right\}\\
  R_1(\arr, \{1,3,5\}) &= \left\{ \lambda \in \CC^6 \Big| \lambda_1 + \lambda_3 + \lambda_5 = 0, \lambda_2=\lambda_4=\lambda_6=0 \right\}\\
   R_1(\arr, \{2,3,4\})&= \left\{ \lambda \in \CC^6 \Big| \lambda_2 + \lambda_3 + \lambda_4 = 0, \lambda_1=\lambda_5=\lambda_6=0 \right\}\\
  R_1(\arr, \{4,5,6\})&= \left\{ \lambda \in \CC^6 \Big| \lambda_4 + \lambda_5 + \lambda_6 = 0, \lambda_1=\lambda_2=\lambda_3=0 \right\}\\
\end{align*}
Let  \( \{f_1, \dots, f_6\} \) be the canonical set of basis vectors for  \( \CC^6 \). Associate each hyperplane  \( H_i \) with the vector  \( f_i \). Then we may write  \( R_1(\arr, \{1,2,6\}) \) as the span of the vectors  \(f_1 - f_2\) and  \( f_2 - f_6 \).
\end{example}

\begin{notation}
  Let \( V \subseteq \CC^n\) be any variety that is the union of linear subspaces. Then, each point  \( v \in V \) may be regarded as a vector  \( \mathbf{v} \in \CC^n \).  Denote by  \( \Span{V}\) the linear subspace of  \( \CC^n \) spanned by the vectors \( \mathbf{v} \in V. \)
\end{notation}

\begin{theorem}\label{thm:charvar-to-resvar}
  Let  \( \arr \) be a central arrangement of  \( n \) hyperplanes such that \[
  V_1^1(\arr) \cong (V_1^1(M_1) \times \mathbf{1}^{n_2}) \cup (\mathbf{1}^{n_1} \times V_1^1(M_2)),
  \] where  \( n=n_1+n_2 \). Then, the resonance variety  \( R_1(\arr) \) decomposes into a union of varieties $R_1 (\arr) = R_1(M_1) \cup R_1(M_2).$ such that the intersection of \( \Span{R_1(M_1)} \) and \(\Span{R_1(M_2)} \) is the trivial vector.
\end{theorem}
\begin{proof} 
  By Theorem 5.2 in \cite{CS-CharVars-MR1692519}, the tangent cone  \( \mathcal{V}_1(\arr) \) of the characteristic variety  \( V_1^1(\arr) \) at the point  \( \mathbf{1} \) coincides with the resonance variety  \( R_1(\arr) \). More explicitly, there is a linear isomorphism  \( \phi \) from the tangent space of  \( \CC^n \) at  \( \mathbf{1} \) to  \( \CC^n \) such that  \( \phi(\mathcal{V}_k(\arr)) = R_1(\arr) \). 

  Let  \( \CC[t_1, \dots, t_{n_1},t_{n_1 +1}, \dots, t_{n_1 + n_2}]\) be a coordinate ring for  \( V = (V_1^1(M_1) \times \mathbf{1}^{n_2}) \cup (\mathbf{1}^{n_1} \times V_1^1(M_2)) \) and let  \( \CC[z_1, \dots, z_{n_1},z_{n_1 +1}, \dots, z_{n_1 + n_2}] \) be a coordinate ring for the tangent space of  \( \mathbf{1} \) in  \( \CC^n. \) Then the variety \(V_1^1(M_1) \times \mathbf{1}^{n_2}\) is defined by an ideal  \( I_1 \in \CC[t_1, \dots, t_{n_1},t_{n_1 +1}, \dots, t_{n_1 + n_2}]  \) generated by the union of an ideal  \( I_1' \in  \CC[t_1, \dots, t_{n_1}]\) and  the set \( \{t_{n_1 +1}-1, \dots, t_{n_1 + n_2}-1\} \). Therefore, the tangent cone  \( \mathcal{V}(M_1) \) of the variety \(V_1^1(M_1) \times \mathbf{1}^{n_2}\) at  \( \mathbf{1} \) is defined by the ideal generated by an ideal \( J_1 \in \CC[z_1, \dots, z_{n_1}] \) and the set  \( \{z_{n_1 + 1}, \dots, z_{n_1 + n_2}\} \). In a similar manner, the tangent cone of  \( \mathbf{1}^{n_1} \times V_1^1(M_2) \) at  \( \mathbf{1} \) denoted by  \( \mathcal{V}(M_2) \) is defined by the ideal generated by an ideal \( J_2 \in \CC[z_{n_1 +1}, \dots, z_{n_1 + n_2}] \) and the set  \( \{z_{1}, \dots, z_{n_1}\} \). Therefore, the tangent cone of  \( V \) at  \( \mathbf{1} \) is given by  \( \mathcal{V}(M_1) \cup \mathcal{V}(M_2) \). From the definition of the ideals generating  \( \mathcal{V}(M_1) \) and \( \mathcal{V}(M_2) \), the varieties are orthogonal. Therefore, the subspaces  \( \Span{\mathcal{V}(M_1)}\) and \(\Span{\mathcal{V}(M_2)}\) intersect only at the origin.

  As  \( V_1^1(\arr) \cong V \), there is a monomial automorphism  \( g \) of  \( \CC^n \) inducing a linear isomorphism  \( g^* \) of tangent cones  such that \( g^*(\mathcal{V}(M_1) \cup \mathcal{V}(M_2)) = \mathcal{V}_1(\arr) \). Combined with the map  \( \phi \), we have \( R_1(\arr)  = \phi(g^*(\mathcal{V}(M_1) \cup \mathcal{V}(M_2))) \). Define  \( R_1(M_i) = \phi(g^*(\mathcal{V}(M_i)))  \). Then, as linear isomorphisms preserve unions and intersections of linear subspaces, the conclusion of the theorem follows.
\end{proof}

 \begin{theorem}
\label{thm:local-component-direct-sum}
Let  \( \arr^* \) be an arrangement of projective lines in  \( \CP^2 \) such that \[
  V_1^1(\carr) \cong (V_1^1(M_1) \times \mathbf{1}^{n_2}) \cup (\mathbf{1}^{n_1} \times V_1^1(M_2)) \cup (\mathbf{1}^{n_1} \times V_1^1(\ZZ))
  \] where the varieties  \( V_1^1(M_i) \) are not trivial. Let $R_1 (\arr) = R_1(M_1) \cup R_1(M_2)$ be the decomposition of the resonance variety from Theorem \ref{thm:charvar-to-resvar}.
If either  \( R_1(M_1) \) or  \( R_1(M_2) \) do not have local components, then  \( \arr^* \) has a decone with a general position partition.
\end{theorem}
\begin{proof}
  Suppose that both \( R_1(M_1) \) and \( R_1(M_2) \) do not have local components. Then there are no multiple points in the arrangement  \( \arr^* \). Deconing  \( \arr^* \) with respect to any line, will result in an arrangement of lines in general position. Thus an arrangement that has a general position partition.

  Without loss of generality, assume \( R_1(M_1) \) has local components and \( R_1(M_2) \) does not have local components. Let  \( \mathcal{B}_1 \) be the set of lines in  \( \arr^* \) containing points that induce local components in \( R_1(M_1) \) and let  \( \mathcal{B}_2 = \arr^* \setminus \mathcal{B}_1 \). If \( \mathcal{B}_1 = \arr^* \), then every line in the arrangement induces a local component contained in  \( R_1(M_1) \). Then  \( \Span{R_1(M_1)}= \triangle = \left\{ \lambda \in \CC^n \Big| \sum\limits_{i=1}^n \lambda_i =0 \right\}\). However,  as \( R_1(\carr) \subseteq \triangle \) \cite{Falk-ArrsAndCohom-MR1629681},   \( R_1(M_2) \) must be the trivial subspace. By hypothesis,  \( V_1^1(M_2) \) is not trivial, therefore its tangent cone and the variety  \( R_1(M_2) \) are not trivial subspaces. Therefore we have a contradiction and may conclude that  \( \mathcal{B}_1 \) and  \( \mathcal{B}_2 \) are not empty. In fact,  \( |\mathcal{B}_1 | \geq 3 \) as it contains at least one local component.

  Any line \( L_1 \in \mathcal{B}_1  \) and  \( L_2 \in \mathcal{B}_2 \) must intersect in a point of multiplicity two. If they intersect in a higher order point, the intersection induces a local component, but the lines in  \( \mathcal{B}_2 \) do not induce local components. Therefore,  \( \mathcal{B}_1   \) and  \( \mathcal{B}_2\) intersect in general position in  \( \CP^2. \)

  Pick any line  \( L \in \mathcal{B}_1 \) and consider  \( \decone{L}^*  \). The images of  \( \mathcal{B}_1   \) and  \( \mathcal{B}_2  \) intersect transversely in  \( \CC^2 \), thus  \( \arr^* \) has a decone with a general position partition.
\end{proof}

\section{Main Theorem}

We now prove Main Theorem 1 from the introduction, rephrased to make use of our terminology:

\begin{theorem}\label{thm:main}
  Let $\arr^*$ be a projective arrangement of lines in $\CP^2$ such that  $\pi_1(M(\arr^*))$ is isomorphic to a direct product of two non-trivial groups. Then  \( \arr^* \) has a decone with a general position partition.
\end{theorem}

\begin{proof} 
  We proceed by contradiction. Suppose that \( \arr^* \) does not have a decone with a general position partition. 

As $\pi_1(M(\arr^*)) \cong B \times C$ where  \( B \) and  \( C \) are not isomorphic to the trivial group, by Theorem \ref{thm:group-to-charvar-union} and Theorem \ref{thm:charvar-to-resvar}, we have that  \( R_1(\carr ) = R_1(M_1) \cup R_1(M_2) \) where the intersection of the subspaces \( \Span{R_1(M_1)} \) and  \( \Span{R_1(M_2)} \) consists only of the trivial vector.

As  \( \arr^*  \) does not have a decone with a general position partition, by Lemma \ref{lem:disconnected} all choices of graphs of Fan-type  \( F \in \mathcal{F}(\arr^*)\) are connected. By Theorem \ref{thm:local-component-direct-sum}, both components  \( R_1(M_1) \) and  \( R_1(M_2) \) must have local components. Therefore, there exists an edge  \( e \in F \) with vertices  \( v,w \) such that  \( R_1^{loc}(v) \subseteq R_1(M_1) \) and \( R_1^{loc}(w) \subseteq R_1(M_2) \).

Since \( \arr^* \) does not have a decone with a general position partition, by the contrapositve of Theorem \ref{thm:norepeats} there is a choice of graph of Fan type \( F \in \mathcal{F}(\arr^*) \) such that for every simple circuit containing the edge  \( e \), the circuit does not contain another edge that lies in the same line as  \( e \). Let $C =\{v,e,w,e_2,z_3,e_3,z_4,\dots, z_m,e_m,v\}$ be such a circuit. 

Let  \( H \) denote the line containing the edge  \( e \) and let  \( H_i \) denote the line containing  \( e_i \) for  \( 2 \leq i \leq m \). (Note that  \( H_i  \) may be the same line as some  \( H_j \) for  \( i \neq j \); however,  \( H \neq H_i \) for all  \( 2 \leq i \leq m \).) Let  \( \{f_i\}_{i=1}^n\) denote the canonical set of basis vectors in  \( \CC^n \). 

Let  \( H \) be associated with the vector  \( f_1 \). Let  \( H_j \) be associated with a vector  \( g_j \) such that  \( g_j \in \{f_i\}_{i=2}^n \). Then we have the following vectors in the local resonance components induced by the vertices of the circuit \( C \).
\begin{align*}
  f_1 - g_m &\in R_1^{loc}(v) \\ 
  g_2 - f_1 &\in R_1^{loc}(w) \\
 g_j - g_{j-1} &\in R_1^{loc}(z_j), \text{ for } 3 \leq j \leq m 
\end{align*}

One can see that \[
(f_1 - g_m) + (g_2 - f_1) + \sum\limits_{j=3}^m  (g_j - g_{j-1}) = 0.
\]
Let  \[ I = \{ i : 3 \leq i \leq m, R_1^{loc}(z_i) \subseteq R_1(M_1)\} \]  \[ J = \{ j : 3 \leq j \leq m, R_1^{loc}(z_j) \subseteq R_1(M_2)\} \] Rearranging the sum so that all vectors in local components contained in  \( R_1(M_1) \) are on the left side of the equals sign and all vectors in local components contained in  \( R_1(M_2) \) are on the right side of the equals sign yields
\[
(f_1 - g_m) + \sum\limits_{i\in I} ( g_i - g_{i-1}) =  - (g_2 - f_1)  - \sum\limits_{j\in J} ( g_j - g_{j-1}).
\]
As  \( g_j \in \{f_i\}_{i=2}^n   \), the  vector \( f_1 \) only appears once on each side of the equality. Therefore, both sides of the equality are non-trivial vectors. Further, \[
(f_1 - g_m) + \sum\limits_{i\in I} ( g_i - g_{i-1})  \in \Span{R_1(M_1)} 
\] 
and \[
-(g_2 - f_1) - \sum\limits_{j\in J} ( g_j - g_{j-1})  \in \Span{R_1(M_2)}. 
\]
Therefore,  \( \Span{R_1(M_1)}  \cap \Span{R_1(M_2)} \neq \{\mathbf{0}\} \), which contradicts Theorem \ref{thm:charvar-to-resvar}.

Hence, \( \arr^* \) does  have a decone with a general position partition.
\end{proof}

We may now prove Main Theorem 2 from the introduction:

\begin{theorem}\label{thm:converse-okasakamoto}
  Let  \( \arr_1 \) and  \( \arr_2 \) be non-empty arrangements in  \( \CC^2 \) such that \[
  \pi_1(M(\arr_1 \cup \arr_2)) \cong \pi_1(M(\arr_1)) \times \pi_1(M(\arr_2)).
  \]
  Then,  the intersection of \( \arr_1 \) and  \( \arr_2 \) consists of  \( |\arr_1| \cdot |\arr_2| \) points of multiplicity two.
\end{theorem}
\begin{proof}
  Let  \( n_i \) be the number of lines in the arrangement  \( \arr_i \). By Theorem \ref{lem:fungroups-rel}, we know that \[
  \pi_1(M(\mathbf{c}(\arr_1 \cup \arr_2))) \cong \pi_1(M(\arr_1)) \times \pi_1(M(\arr_2)) \times \ZZ.
  \] 
  Using Theorem \ref{thm:group-to-charvar-union}, the direct product
  structure of the fundamental group allows us to decompose the
  characteristic variety as
\begin{align*}
  V_1^1(M(\mathbf{c}(\arr_1 \cup \arr_2))) & \cong
  (V_1^1(\pi_1(M(\arr_1))) \times \mathbf{1}^{n_2 + 1})  \\ &\cup (\mathbf{1}^{n_1+ 1}  \times V_1^1(\pi_1(M(\arr_1))) \times 1) \\&\cup  ( \mathbf{1}^{n_1 + n_2 + 1}).
\end{align*}
As  \( \pi_1(M(\mathbf{c}\arr_i)) \cong \pi_1(M(\arr_i)) \times \ZZ \), we may identify the characteristic varieties associated to each group, ie \[
  V_1^1(\pi_1(M(\mathbf{c}\arr_i))) = V_1^1(\pi_1(M(\arr_i)) \times \ZZ) = V_1^1(\pi_1(M(\arr_i))) \times \{1\}. 
  \]  Therefore, by Theorem 5.2 in \cite{CS-CharVars-MR1692519}, we may identify the resonance variety  \(R_1(\mathbf{c}\arr_i)\) with the tangent cone of the variety  \(  V_1^1(\pi_1(M(\arr_i))) \times \{1\} \) at  \( \mathbf{1} \). We may also identify  \( \mathbf{c}\arr_i \) as subarrangements of  \( \mathbf{c}(\arr_1 \cup \arr_2) \), where  \( \mathbf{c}\arr_1 \) and  \( \mathbf{c}\arr_2 \) have the hyperplane  \( H_{\infty} \) that was added to the cone over  \( \arr_1 \cup \arr_2 \) in common.
 
  Using Theorem \ref{thm:charvar-to-resvar} and the identifications given above, we see that \[
  R_1(\mathbf{c}(\arr_1 \cup \arr_2)) \cong R_1(\mathbf{c}\arr_1) \cup R_1(\mathbf{c}\arr_2)\] 
   and the intersection of \(R_1(\mathbf{c}\arr_1)\) and \(R_1(\mathbf{c}\arr_2)\) is the origin. As such, the varieties \(R_1(\mathbf{c}\arr_1)\) and \(R_1(\mathbf{c}\arr_2)\) do not have any non-trivial components in common.

   Suppose by way of contradiction, that the arrangements  \( \arr_1 \) and  \( \arr_2 \) do not intersect in \( |\arr_1| \cdot |\arr_2| \) points of multiplicity two. Then there exists lines in the arrangements,  \( H_1 \in \arr_1 \) and  \( H_2 \in \arr_2 \), such that the lines intersect in a point of multiplicity greater than two in the arrangement  \( \arr_1 \cup \arr_2 \) or are parallel. If the lines are parallel, then the hyperplanes in  \( \mathbf{c}(\arr_1 \cup \arr_2) \) corresponding to  \( H_1 \) and  \( H_2 \) intersect with the hyperplane  \( H_{\infty} \) added to the arrangement \( \mathbf{c}(\arr_1 \cup \arr_2) \)  in a codimension two subspace  \( X \) of  \( \CC^3 \). Likewise, if the lines intersect in a higher order point, then the point corresponds to a codimension two subspace  \( X \) in the cone \( \mathbf{c}(\arr_1 \cup \arr_2) \). By abuse of notation, let  \( H_\infty \) denote a line different from  \( H_1 \) and  \( H_2 \) such that  \( X \subseteq H_\infty \). In either case, the subspace  \( X \) induce a non-trivial local component  \( R_1^{loc}(X) \) of the resonance variety  \( R_1(\mathbf{c}(\arr_1 \cup \arr_2))  \).

   Let  \( \{e_{H}\}_{H \in \mathbf{c}(\arr_1 \cup \arr_2)} \) be a basis for the vector space  \( \CC^{n_1 + n_2 + 1} \) containing the resonance variety  \( R_1(\mathbf{c}(\arr_1 \cup \arr_2)) \). Then the vectors  \( e_{H_1} - e_{H_{\infty}} \) and  \( e_{H_2} - e_{H_{\infty}} \) are both contained in the local component \( R_1^{loc}(X). \) However,  \( e_{H_2} - e_{H_{\infty}}  \notin R_1(\mathbf{c}\arr_1)\) and \( e_{H_1} - e_{H_{\infty}}  \notin R_1(\mathbf{c}\arr_2)\) as \( H_2 \notin \arr_1 \) and \( H_1 \notin \arr_2\).  Therefore, the connected component \( R_1^{loc}(X) \)  is not contained in either \(R_1(\mathbf{c}\arr_1)\) or \(R_1(\mathbf{c}\arr_2)\).  But this is a contradiciton as  \( R_1(\mathbf{c}(\arr_1 \cup \arr_2) ) = R_1(\mathbf{c}\arr_1) \cup R_1(\mathbf{c}\arr_2) \) and  \( R_1(\mathbf{c}\arr_1) \cap R_1(\mathbf{c}\arr_2) = \{\mathbf{0}\}. \)

Therefore, the intersection of \( \arr_1 \) and  \( \arr_2 \) consists of  \( |\arr_1| \cdot |\arr_2| \) points of multiplicity two and the theorem is proven.
\end{proof}

We are left with the following question:

\textbf{Question: } For what classes of algebraic curves does the converse of the theorem
of Oka and Sakamoto hold?

\section{Applications and Examples}

\begin{figure}[!h]
		\begin{center} 
			\begin{tikzpicture}[scale=0.7]
				
				\draw (-3,1) to (3,1);
				\draw (-3,-1) to (3,-1);
				\draw (1,-3) to (1,3);
				\draw (-1,-3) to (-1,3);
				\draw (-3,-3) to (3,3);
 \draw[rounded corners=5mm] (-3.4,0) -- (-3.4,-3.4) -- (3.4,-3.4) --(3.4,3.4) -- (-3.4,3.4)--(-3.4,0); 
\draw (-2.7,1.5)  node {$H_1$};
\draw (-2.7,-0.5)  node {$H_2$};
\draw (-2.7,-2)  node {$H_3$};
\draw (-1.5,-2.7)  node {$H_4$};
\draw (0.5,-2.7)  node {$H_5$};
\draw (2.7,-2.7)  node {$H_6$};
			\end{tikzpicture}
			\caption{The real part of the projective arrangement defined by $Q(\Arr^*) = xyz(x-y)(x-z)(y-z)$. }
			\label{fig:braid-arrangement}
		\end{center}
	\end{figure}
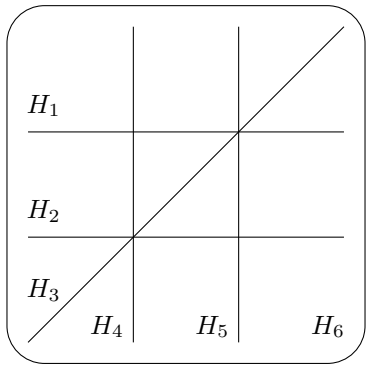

\begin{example} Let  \( \arr_3^* \) be the projective arrangement defined by  \( Q(\Arr_3^*) = xyz(x-y)(x-z)(y-z) \). The real part of the arrangement is depicted in Figure \ref{fig:braid-arrangement}, where the line labelled with  \( H_6 \) is the ``line at infinity.'' One should also recall that the parallel lines meet at infinity. The local components of the resonance variety of the cone over  \( \arr^* \) are shown in Example \ref{ex:braid-arr-resvars}. Let  \( \{f_1,\dots,f_6\} \) denote the canonical basis for  \( \CC^6 \) as a vector space. Then each local resonance variety may be regarded as a span of vectors as follows:
\begin{align}
  R_1(\carr^*, \{1,2,6\})&= \Span{f_1 - f_6, f_2 - f_6}  & R_1(\carr^*, \{1,3,5\})  &= \Span{f_1 - f_5, f_3 - f_5} \\
   R_1(\carr^*, \{2,3,4\})&= \Span{f_2 - f_4, f_3 - f_4} & R_1(\carr^*, \{4,5,6\})&= \Span{f_4 - f_6, f_5 - f_6}  \\
\end{align}
By a simple exercise in linear algebra, one can see that there does not exist a decomposition of  \( R_1(\carr^*) \) into varieties  \( A \) and  \( B \) such that  \( \Span{A} \cap \Span{B} = \{0\} \). Therefore by combining the converse of Theorem \ref{thm:charvar-to-resvar} and Corollary \ref{cor:prod-union}, we have that  \( \pi_1(M(\arr^*)) \) is not isomorphic to a non-trivial direct product of groups.
\end{example}

\begin{example} Let  \( \arr\) be the arrangement defined by intersecting a generic hyperplane with the arrangement defined by $Q(\arr_3) = xyz(x-y)(x-z)(y-z)$ in  \( \CC^3 \). Figure \ref{fig:gen2sec-braid} depicts the real part the arrangement  \( \arr. \) We note that the fundamental group \( \pi_1(M(\arr)) \) is isomorphic to the pure braid group on four strands  \( P_4 \). It is well known that  \( P_4 \cong (\FF_3 \rtimes \FF_2) \times \ZZ \). Therefore, by Theorem \ref{thm:main} the projective completion of  \( \arr \) has a line we may decone the arrangement with respect to so that the decone has a general position partition. In this case, we may see in Figure \ref{fig:proj-braid} that if we use the line  \( H_6 \) as the ``line at infinity'' then we may partition the decone into the sets  \( \{H_7\}\) and  \( \{H_1,H_2,H_3,H_4,H_5\} \) to obtain a general position partition.
  
	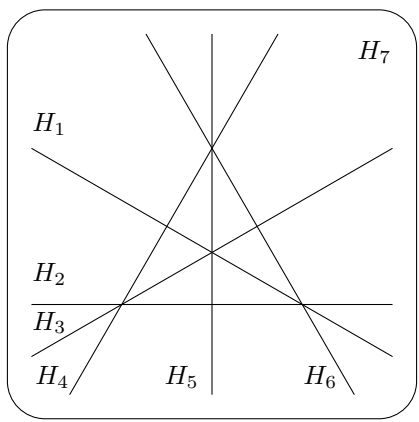
\begin{figure}[!h]
		\begin{center} 
			\begin{tikzpicture}[scale=0.8]
				
				\draw (-3,-1.5) to (3,-1.5);
\draw(-3,-2.3660254) to (3,1.09807621);
\draw(3,-2.3660254) to (-3,1.09807621);
\draw(0,-3) to (0,3);
\draw(-2.3660254,-3) to (1.09807621 ,3);
\draw(2.3660254,-3) to (-1.09807621 ,3);
 \draw[rounded corners=5mm] (-3.4,0) -- (-3.4,-3.4) -- (3.4,-3.4) --(3.4,3.4) -- (-3.4,3.4)--(-3.4,0); 
\draw (-2.7,1.5)  node {$H_1$};
\draw (-2.7,-1)  node {$H_2$};
\draw (-2.7,-1.8)  node {$H_3$};
\draw (-2.65,-2.7)  node {$H_4$};
\draw (-0.5,-2.7)  node {$H_5$};
\draw (1.8,-2.7)  node {$H_6$};
\draw (2.7,2.7)  node {$H_7$};
				 
			\end{tikzpicture}
			\caption{The real part of the projectivization of a generic two-dimensional section of the arrangement defined by $Q(x,y,z) = xyz(x-y)(x-z)(y-z)$. }
			\label{fig:gen2sec-braid}
		\end{center}
	\end{figure}
\end{example}

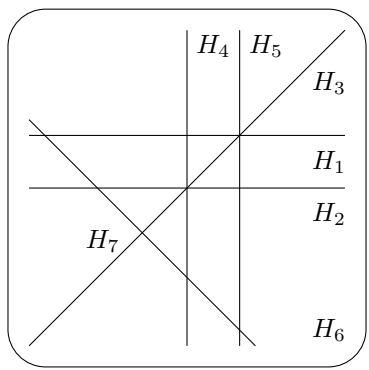
\begin{figure}[!h]
		\begin{center} 
			\begin{tikzpicture}[scale=0.7]
				
				\draw (-3,1) to (3,1);
				\draw (-3,0) to (3,0);
				\draw (1,-3) to (1,3);
				\draw (0,-3) to (0,3);
				\draw (-3,-3) to (3,3);
\draw (-3,1.3) to (1.3,-3);
 \draw[rounded corners=5mm] (-3.4,0) -- (-3.4,-3.4) -- (3.4,-3.4) --(3.4,3.4) -- (-3.4,3.4)--(-3.4,0); 
\draw (2.7,0.5)  node {$H_1$};
\draw (2.7,-0.5)  node {$H_2$};
\draw (2.7,2)  node {$H_3$};
\draw (0.5,2.7)  node {$H_4$};
\draw (1.5,2.7)  node {$H_5$};
\draw (2.7,-2.7)  node {$H_6$};
\draw (-1.6,-1)  node {$H_7$};			 
			\end{tikzpicture}
			\caption{The real part of the projectivization of a generic two-dimensional section of the arrangement defined by $Q(\Arr) = xyz(x-y)(x-z)(y-z)$, with the line  \( H_6 \) as the ``line at infinity.'' }
			\label{fig:proj-braid}
		\end{center}
	\end{figure}

	\bibliographystyle{amsalpha} 
\providecommand{\bysame}{\leavevmode\hbox to3em{\hrulefill}\thinspace}
\providecommand{\MR}{\relax\ifhmode\unskip\space\fi MR }
\providecommand{\MRhref}[2]{%
  \href{http://www.ams.org/mathscinet-getitem?mr=#1}{#2}
}
\providecommand{\href}[2]{#2}

\end{document}